\documentclass[11pt]{amsart}
\usepackage{amssymb}
\usepackage{latexsym} 
\usepackage{srcltx} 
\usepackage{esint}
\usepackage{color}

\def\nc{\newcommand}

\def\om{\omega}

\nc\pa{\partial}

\def\longequals{\mathbin{=\kern-2pt=}}

\nc\CC{\mathbb{C}}
\nc\RR{\mathbb{R}}
\nc\QQ{\mathbb{Q}}
\nc\ZZ{\mathbb{Z}}
\nc\NN{\mathbb{N}}
\nc{\supp}{\mathop{\mathrm{supp}}}
\nc{\ts}[1]{\widetilde{{#1}}}

\nc\m[1]{\left| #1\right|}
\nc\norm[1]{\left\| #1\right\|}
\nc\dH[1]{\dot{H}^{#1}}

\newcommand{\Lp}[2]{L^{#1}\left(#2\right)}
\newcommand{\Ci}[1]{\mathcal{C}^{\infty}_{0}(#1)}

\def\longequals{\mathbin{=\kern-2pt=}}

\newtheorem{theorem}{Theorem}[section]
\newtheorem{lemma}[theorem]{Lemma}
\newtheorem{corollary}[theorem]{Corollary}
\newtheorem{proposition}[theorem]{Proposition}
\newtheorem{definition}[theorem]{Definition}
\numberwithin{equation}{section}

\newcommand{\ds}{\displaystyle}

\begin{document}

\title[Th 3D Navier-Stokes equations]
{Local energy bounds and $\epsilon$-regularity criteria for the 3D Navier-Stokes  system}

\author[Cristi Guevara]{Cristi Guevara}
\address{Department of Mathematics,
Louisiana State University,
303 Lockett Hall, Baton Rouge, LA 70803, USA.}
\email{cguevara@lsu.edu}

\author[Nguyen Cong Phuc]
{Nguyen Cong Phuc$^*$}
\address{Department of Mathematics,
Louisiana State University,
303 Lockett Hall, Baton Rouge, LA 70803, USA.}
\email{pcnguyen@math.lsu.edu}

\thanks{$^*$Supported in part by Simons Foundation, award number: 426071}


\begin{abstract} The system of three dimensional Navier-Stokes equations is considered. We obtain some new local energy bounds that enable us to improve several $\epsilon$-regularity criteria. They key idea here is to view the `head pressure' as a signed distribution belonging to certain fractional Sobolev space of negative order. This allows us to capture the oscillation 
of the pressure  in our criteria.
\end{abstract}

\maketitle

\section{Introduction}
We are concerned with the three dimensional Navier-Stokes system
\begin{equation}\label{NSE}
\partial_t u -\Delta u + u\cdot\nabla u + \nabla p=0, \quad {\rm div}\, u=0,
\end{equation}
where $u=u(x,t)=(u_1(x,t), u_2(x,t), u_3(x,t))$ is the velocity of the fluid and the scalar function $p=p(x,t)$ is its pressure.
The system \eqref{NSE} also comes with certain boundary and initial conditions but we shall not specify them here.

Since the seminal work of Leray \cite{Ler} and Hopf \cite{Hop}, it is known that  there exist global in time weak solutions with finite energy to the initial-boundary value problem associated to \eqref{NSE}. Such solutions are now called Leray-Hopf weak solutions. However, the questions of regularity and uniqueness of  
Leray-Hopf weak solutions are still unresolved.

To investigate the regularity of system \eqref{NSE}, in the fundamental paper \cite{CKN}, Caffarelli-Kohn-Nirenberg introduced the notion of suitable weak solutions. They obtained existence as well as partial regularity  for suitable weak solutions. Their fundamental result states that the one-dimensional parabolic Hausdorff measure of the possible singular set of 
suitable weak solutions is zero (see also \cite{CL}). The proof of this partial regularity result is based on the following $\epsilon$-regularity criterion: there is an $\epsilon>0$ such that if $u$ is a suitable weak solution in $Q_1=B_1(0)\times(-1,0)$ and satisfies
\begin{equation*}
\limsup_{r\rightarrow 0} \frac 1 r \int_{Q_r} |\nabla u|^2 dy ds\leq \epsilon,
\end{equation*}
then $u$ is regular at the point $(0,0)$, i.e., $u\in L^\infty(Q_r)$ for some $r>0$. Here we write $Q_r=B_r(0)\times(-r^2,0)$. 

In turn the proof of this $\epsilon$-regularity criterion is based on another one that involves both $u$ and $p$ but requires the smallness at only one scale:
\begin{theorem}[\cite{CKN}]\label{q=3} There exists an $\epsilon>0$ such that if $u$ is a suitable weak solution in $Q_1$ and satisfies 
\begin{equation}\label{Linsmall}
\int_{Q_1} (|u|^3 + |p|^{3/2}) dy ds\leq \epsilon,
\end{equation}
then $u\in  L^\infty(Q_{1/2})$. 
\end{theorem}

Theorem \ref{q=3} was first proved in \cite[Proposition 1]{CKN} in a slightly more general form, namely, the smallness condition \eqref{Linsmall} is replaced by the condition
\begin{equation}\label{CKNsmall}
\int_{Q_1} (|u|^3 + |p| |u|) dy ds + \int_{-1}^{0} \norm{p}_{L^1(B_1(0))}^{\frac 54}ds \leq \epsilon.
\end{equation}

The proof presented in \cite{CKN} is based on an inductive argument that goes back to Scheffer \cite{Sche2}. Later Lin \cite[Theorem 3.1]{Lin} gave a new proof based on a compactness argument. In fact, he showed that under \eqref{Linsmall} the solution is H\"older continuous with respect to the space-time parabolic metric on the closure of $B_{1/2}(0)\times(-1/4,0)$. See also  \cite[Lemma 3.1]{LS}. We mention that Theorem \ref{q=3} has also been used as an important tool in many other papers such as \cite{NRS, GKT, ESS1, WZ, Phuc}, etc.

A more constructive approach to Theorem \ref{q=3} can be found in \cite{Vass} in which Vasseur used De Giorgi iteration technique to obtain it in the following form.
\begin{theorem}[\cite{Vass}] \label{VassThm}
For each $p>1$ there exists an $\epsilon(p)>0$ such that if $u$ is a suitable weak solution in $Q_1$ and satisfies 
\begin{equation*}
 \sup_{t\in[-1,0]} \int_{B_1(0)} |u(x,t)|^2 dx + \int_{Q_1} |\nabla u|^2 dy ds +\int_{-1}^0 \norm{p}_{L^1(B_1(0))}^p ds   \leq \epsilon(p),
\end{equation*}
then $u\in  L^\infty(Q_{1/2})$. 
\end{theorem}

It is not hard to see from the generalized energy inequality (see Definition \ref{SWS} below) and a simple covering argument that Theorem   \ref{VassThm} indeed implies Theorem   \ref{q=3}.

We now state another related $\epsilon$-regularity criterion that was obtained and used in \cite[Proposition 5.1]{WZ}.
\begin{theorem}[\cite{WZ}] \label{WZThm}
There exists an $\epsilon>0$ such that if $u$ is a suitable weak solution in $Q_1$ and satisfies 
\begin{eqnarray}\label{WZsmall}
 \sup_{t\in[-1,0]} \int_{B_1(0)} |u(x,t)|^2 dx &+& \int_{-1}^0\norm{u}_{L^4(B_1(0))}^2 ds+\\
 && +\, \int_{-1}^0\norm{p}_{L^2(B_1(0))} ds    \leq \epsilon, \nonumber 
\end{eqnarray}
then $u\in  L^\infty(Q_{1/2})$. 
\end{theorem}

Finally, we mention yet another $\epsilon$-regularity result that was obtained by the second named author  in \cite[Proposition 3.2]{Phuc}.
\begin{theorem}[\cite{Phuc}] \label{PhThm}
There exists an $\epsilon>0$ such that if $u$ is a suitable weak solution in $Q_1$ and satisfies 
\begin{equation*}
 \int_{-1}^0\norm{u}_{L^\frac{12}{5}(B_1(0))}^4 ds +\, \int_{-1}^0\norm{p}_{L^\frac{6}{5}(B_1(0))}^{2} ds    \leq \epsilon, 
\end{equation*}
then $u\in  L^\infty(Q_{1/2})$. 
\end{theorem}

The goal of this paper is to sharpen and unify the results obtained in Theorems \ref{q=3}-\ref{PhThm}. Our first result says 
that in fact one can take $p=1$ in Theorem \ref{VassThm}, i.e., we prove
\begin{theorem} \label{VassThmp=1}
There exists an $\epsilon>0$ such that if $u$ is a suitable weak solution in $Q_1$ and satisfies 
\begin{equation*}
 \sup_{t\in[-1,0]} \int_{B_1(0)} |u(x,t)|^2 dx + \int_{Q_1} |\nabla u|^2 dy ds +\int_{-1}^0 \norm{p}_{L^1(B_1(0))} ds   \leq \epsilon,
\end{equation*}
then $u\in  L^\infty(Q_{1/2})$. 
\end{theorem}
This theorem implies that in the condition \eqref{CKNsmall} of Caffarelli, Kohn, and Nirenberg one can replace the power $\frac54$ in the pressure term by 1. Our next result reads as follows.
\begin{theorem} \label{PhThm-alpha-beta} Let $\alpha\in[6/5, 2]$ and  $\beta=\frac{4\alpha}{7\alpha -6}\in[1, 2]$.
There exists an $\epsilon>0$ such that if $u$ is a suitable weak solution in $Q_1$ and satisfies 
\begin{equation*}
 \int_{-1}^0\norm{u}_{L^{2\alpha}(B_1(0))}^{2\beta} ds +\, \int_{-1}^0\norm{p}_{L^\alpha(B_1(0))}^{\beta} ds  
  \leq \epsilon, 
\end{equation*}
then $u\in  L^\infty(Q_{1/2})$. 
\end{theorem}
The case $(\alpha,\beta)=(18/13, 3/2)$ gives a spatial improvement of Theorem \ref{q=3}, whereas the case $(\alpha,\beta)=(3/2, 4/3)$ gives a time improvement. Kukavica \cite[p. 2845]{Kuka} mentioned the issue whether the number 3 in \eqref{Linsmall} can be replaced by some $q<3$. Indeed, this is the case if we take $q=2\alpha=2\beta=20/7$. This gives both space and time improvement of Theorem  \ref{q=3}. Moreover, Theorems \ref{WZThm} and 
\ref{PhThm} are special end-point cases of Theorem \ref{PhThm-alpha-beta}, with $\alpha=2$ and $\alpha=6/5$, respectively . In fact, it also implies that the first term in 
condition \eqref{WZsmall} can be dropped.

Theorem \ref{PhThm-alpha-beta} is a consequence of the folllowing result.
\begin{theorem} \label{PhThm-sigma} Let $\sigma\in[0, 1]$. 
There exists an $\epsilon>0$ such that if $u$ is a suitable weak solution in $Q_1$ and satisfies 
\begin{equation*}
 \int_{-1}^0\norm{|u|^2}_{L^{-\sigma,2}(B_1(0))}^{\frac{2}{2-\sigma}} ds +\, \int_{-1}^0\norm{p}_{L^{-\sigma, 2}(B_1(0))}^{\frac{2}{2-\sigma}} ds  
  \leq \epsilon, 
\end{equation*}
then $u\in  L^\infty(Q_{1/2})$. 
\end{theorem}
The space $L^{-\sigma,2}(B_1(0))$ is the dual of the space of functions $f$ in the homogeneous Sobolev space $\dH{\sigma}(\RR^3)$  such that ${\rm supp} f\subset \overline{B_{1}(0)}$. We have $L^{0,2}(B_1(0))= L^2(B_1(0))$. Interestingly, unlike the norm $\norm{p}_{L^\alpha(B_1(0))}$, for $\sigma\in(0,1]$ the norm $\norm{p}_{L^{-\sigma, 2}(B_1(0))}$ 
can `capture' the oscillation of $p$. Namely, it may happen that there exists $f\in L^{-\sigma, 2}(B_1(0))\cap L^1(B_1(0))$ but 
$|f|\not\in L^{-\sigma, 2}(B_1(0))$. In the case $\sigma =1$, one can take for example the function $f(x)=|x|^{-\epsilon-s} \sin(|x|^{-\epsilon})$  
with $s=2.4$ and $\epsilon=0.2$. See also the recent paper \cite{PhTo} for this kind of example in the context of $(BV)^*$, the dual of space of functions of bounded variation.
We mention that by Lemma \ref{bound-sigma} below, this theorem implies Theorem  \ref{PhThm-alpha-beta}.

The proof of Theorem \ref{PhThm-sigma} is based on Theorem \ref{VassThmp=1} and the following new local energy bounds for suitable weak solutions.
\begin{theorem}\label{ENB} Let $\sigma\in[0,1]$. There exists a constant $C>0$ such that for any suitable weak solution $(u,p)$ in $Q_1$ we have
\begin{eqnarray*}
\lefteqn{\sup_{- \frac{1}{4}\leq t\leq 0} 2\int_{B_{1/2}(0)} |u(x,t)|^2 dx + 2\int_{Q_{1/2}} |\nabla u(x,s)|^2 dxds}\\
&\leq& C \left( \int_{-1}^{0} \norm{|u^2|}_{L^{-\sigma,2}(B_1(0))}^{\frac{2}{2-\sigma}} ds\right)^{\frac{2-\sigma}{2}}+\\
&& +\, C \left( \int_{-1}^{0} \norm{|u^2| +2p}_{L^{-\sigma,2}(B_1(0))}^{\frac{2}{2-\sigma}} ds\right)^{2-\sigma}.
\end{eqnarray*}
\end{theorem}
For this result at every point and every scale we refer to  Proposition \ref{ABcontrol0} below. See also Proposition \ref{ABcontrol1b}.

\section{ Preliminaries} 

Throughout the paper we use the following notations for balls and parabolic cylinders:
$$B_r(x)=\{y\in \RR^3: |x-y|<r\}, \quad x\in\RR^3, \, r>0,$$
and
$$Q_r(z)=B_r(x)\times (t-r^2, t) \quad {\rm with~} z=(x,t).$$

The homogeneous Sobolev space $\dH{\sigma}(\RR^3),$ $\sigma\in \RR$, is the space of temper distributions $f$ for which  $\norm{f}_{\dH{\sigma}(\RR^3)}<+\infty$. Here we define
\begin{align*}
\norm{f}_{\dH{\sigma}(\RR^3)} =\left(\int_{\RR^3} |\xi|^{2\sigma}|\hat{f}(\xi)|^2 d\xi\right)^{\frac{1}{2}},   \qquad \sigma\in \RR.
\end{align*}
 The space $L^{\sigma,2}(B_r(x)):= 
 \left\{ f \in \dH{\sigma}(\RR^3) : \supp f\subset \overline{B_r(x)}\right\},$
 and its corresponding the dual space 
 is denoted by $L^{-\sigma,2}(B_r(x))$.

The following scaling invariant  quantities will be employed:

$$A(z_0, r)=A(u, z_0, r)= \sup_{t_0-r^2\leq t\leq t_0} r^{-1}\int_{B_r(x_0)} |u(x,t)|^2 dx,$$
$$B(z_0, r)=B(u, z_0, r)= r^{-1}\int_{Q_r(z_0)} |\nabla u(x,t)|^2 dxdt,$$
$$C_{\sigma}(z_0, r)=C_{\sigma}(u, z_0, r)=r^{-\frac{3}{2-\sigma}} \int_{t_0-r^2}^{t_0}\norm{|u|^2}^{\frac{2}{2-\sigma}}_{L^{-\sigma,2}(B_r(x_0))} dt,$$
$$C_{\alpha,\beta}(z_0, r)=C_{\alpha,\beta}(u, z_0, r)=r^{-\frac{3\beta}{2}}\int_{t_0-r^2}^{t_0} \norm{u}_{L^{2\alpha}(B_r(x_0))}^{2\beta} dt,$$
$$D_{\sigma}(z_0, r)=D_{\sigma}(u, z_0, r)=r^{-\frac{3}{2-\sigma}} \int_{t_0-r^2}^{t_0}\norm{p}^{\frac{2}{2-\sigma}}_{L^{-\sigma,2}(B_r(x_0))} dt,$$
$$D_{\alpha,\beta}(z_0, r)=D_{\alpha,\beta}(p,z_0, r)= r^{-\frac{3\beta}{2}}\int_{t_0-r^2}^{t_0} \norm{p}_{L^{\alpha}(B_r(x_0))}^{\beta} dt.$$

We now recall the  the notion of suitable  weak solutions that was first introduced in Caffarelli-Kohn-Nirenberg \cite{CKN}.  Here we use the version of F.-H. Lin \cite{Lin} that imposes the 3/2 space-time integrability condition on the pressure. 

\begin{definition}\label{SWS} Let $\om$ be an open set in $\RR^3$ and let $-\infty<a< b< \infty$. We say that a pair $(u, p)$ is a suitable weak solution 
to the Navier-Stokes equations in $Q=\om\times (a, b)$ if the following conditions hold:

{\rm (i)}   $u\in L^\infty(a, b; L^2(\om))\cap L^2(a, b; W^{1,\, 2}(\om)) {\rm ~and~} p\in L^{3/2}(\om\times (a, b));$

{\rm (ii)} $(u,p)$  satisfies the Navier-Stokes equations in the sense of distributions. That is, 

$$\int_{a}^{b}\int_{\om} \left\{-u \, \psi_{t} + \nabla u : \nabla \psi - (u\otimes u):\nabla \psi -  p \, {\rm div}\, \psi \right\} dxdt =0$$ 
for all vector fields $\psi\in C_0^{\infty}(\om\times(a,b); \RR^3)$, and
$$\int_{\om\times \{t\}} u(x,t)\cdot \nabla \phi(x) \, dx=0$$
for a.e. $t\in (a, b)$ and  all real valued functions $\phi\in C_0^{\infty}(\om)$;

{\rm (iii)} $(u,p)$  satisfies the local generalized energy inequality
\begin{eqnarray*}
\lefteqn{ \int_{\om}|u(x, t)|^2 \phi(x,t) dx + 2\int_{a}^t\int_{\om} |\nabla u|^2 \phi(x, s) dxds} \\
&\leq& \int_{a}^t\int_{\om} |u|^2 (\phi_t +\Delta \phi) dx ds + \int_{a}^t\int_{\om}(|u|^2 + 2p)u\cdot \nabla \phi dx ds
\end{eqnarray*}
for a.e. $t\in (a, b)$ and any nonnegative function $\phi \in C_0^{\infty}(\RR^3\times\RR)$ vanishing in a neighborhood of the parabolic boundary 
$\partial'Q=\om\times\{t=a\} \cup \partial\om\times [a, b]$.
\end{definition}

We next state several lemmas that are needed in this paper.

\begin{lemma}\label{GNIne} Given $f\in \dH{s_{_0}} \cap \dH{s_{_1}},$ $s_0,s_1 \in \RR$ and $0<\theta<1$, the following Gagliardo-Nirenberg  type inequality  holds 
$$\norm{f}_{\dH{s}}\leq \norm{f}_{\dH{s_{_0}}}^{1-\theta}\norm{f}_{\dH{s_{_1}}}^{\theta}$$
with  $s=(1-\theta)s_0+\theta s_1.$
\end{lemma}
The proof of this lemma simply follows from H\"older's inequality.

\begin{lemma}\label{bound-sigma} For any ball $B_r(x)\subset\RR^3$ and any number $\sigma\in [0, \frac 3 2)$ one has that  
$L^{\frac{6}{3+2\sigma}}(B_r(x))\subset L^{-\sigma,2}(B_r(x))$ and
$$\norm{f}_{L^{-\sigma,2}(B_r(x))}\leq C \norm{f}_{L^{\frac{6}{3+2\sigma}}(B_r(x))}.$$
\end{lemma}

\begin{proof} Observe that
$$\norm{f}_{L^{-\sigma,2}(B_r(x))}=\sup_{\varphi}\int_{B_r(x)} f(y)\varphi(y)dy,$$
where the $\sup$ is taken over  $\varphi\in L^{\sigma, 2}(B_r(x))$ such that $\norm{\varphi}_{\dH{\sigma}(\RR^3)}\leq1$. Thus by H\"older and Sobolev's inequalities we find
\begin{align*}
&\norm{f}_{L^{-\sigma,2}(B_r(x))}\leq\\
&\qquad\leq\sup_{\varphi}\left(\int_{B_r(x)} |f(x)|^{\frac{6}{3+2\sigma}}dx\right)^{\frac{3+2\sigma}6} \left(\int_{B_r(x)} |\varphi(y)|^{\frac{6}{3-2\sigma}}dy\right)^{\frac{3-2\sigma}6} \\
&\qquad\leq C \norm{f}_{L^{\frac{6}{3+2\sigma}}(B_r(x))} \sup_{\varphi} \norm{\varphi}_{\dH{\sigma}(\RR^3)} \leq  C \norm{f}_{L^{\frac{6}{3+2\sigma}}(B_r(x))}. 
\end{align*}
\end{proof}

A proof of the following lemma can be found in \cite[Lemma 6.1]{Giu}.

\begin{lemma}\label{Giusti-lem}
Let $I(s)$ be a bounded nonnegative function in the interval $[R_1, R_2]$. Assume that for every $s, \rho\in [R_1, R_2]$ and  $s<\rho$ we have 
$$I(s)\leq [A(\rho-s)^{-\alpha} +B(\rho-s)^{-\beta} +C] +\theta I(\rho)$$
with  $A, B, C\geq 0$, $\alpha>\beta>0$ and $\theta\in [0,1)$. Then there holds
$$I(R_1)\leq c(\alpha, \theta) [A(R_2-R_1)^{-\alpha} +B(R_2-R_1)^{-\beta} +C].$$

\end{lemma}

We shall also need the following Sobolev interpolation inequality (see, e.g., (1.2) of \cite{LS}).
\begin{lemma} \label{Sob-inte} Let $B_r\subset \RR^3$. For any function $u\in W^{1,2}(B_r)$ such that $\int_{B_r} u dx=0$ and any $q\in [2, 6]$, it holds that 
$$\int_{B_r} |u|^q dx \leq C(q) \left(\int_{B_r} |\nabla u|^2 dx\right)^{3q/4-3/2} \left(\int_{B_r} |u|^2 dx\right)^{-q/4+3/2}.$$ 
\end{lemma}

Lemma \ref{Sob-inte} implies the following well-known result  (see, e.g., \cite[Lemma 5.1]{LS}).
\begin{lemma}\label{boundC1}
Let $u(x,t)$ be a function in $Q_{\rho}(z_0)$ for some $\rho>0$. Then for any $r\in (0, \rho]$ we have
$$r^{-2}\int_{Q_r(z_0)} |u|^3 dx dt\leq C \Big(\frac{\rho}{r} \Big)^3 A(z_0,\rho)^{3/4} B(z_0,\rho)^{3/4} + C \Big(\frac{r}{\rho} \Big)^3 A(z_0,\rho)^{3/2}.$$
\end{lemma}

\section{Local energy estimates}

We prove Theorem \ref{ENB} in this section. We will do it at every point and every scale. The proof employs the idea of viewing the `head pressure' $\frac12 |u|^2+p$ as a signed distribution in $L^{-\sigma,2}$.
\begin{proposition}\label{ABcontrol0}  Suppose that $(u,p)$ is a suitable weak solution 
to the Navier-Stokes equations in  $Q_r(z_0)$. Then it holds that 
\begin{eqnarray*}
A(z_0, r/2) + B(z_0, r/2) &\leq& C\, C_{\sigma}(z_0, r)^{\frac{2-\sigma}{2}}+\\
&+&   C \left[ r^{\frac{-3}{2-\sigma}} \int_{t_0-r^2}^{t_0} \norm{|u|^2 + 2p}_{L^{-\sigma,2}(B_r(x_0))}^{\frac{2}{2-\sigma}} dt\right]^{{2-\sigma}}
\end{eqnarray*}
for any $\sigma\in [0,1]$.
\end{proposition}

\begin{proof}
For $z_0=(x_0, t_0)$ and $r>0$, we consider the cylinders $$Q_s(z_0)=B_s(x_0)\times (t_0-s^2, t_0)\subset Q_\rho(z_0)=B_\rho(x_0)\times (t_0-\rho^2, t_0),$$
where $r/2\leq s<\rho\leq r$.

Let $\phi(x,t)=\eta_1(x)\eta_2(t)$ where $\eta_1\in  C_0^{\infty}(B_\rho(x_0))$, $0\leq \eta_1\leq 1$ in $\RR^n$, $\eta_1\equiv 1$ on $B_s(x_0)$, and $$|\nabla^{\alpha} \eta_1|\leq \frac{c}{(\rho-s)^{|\alpha|}}$$ for all multi-indices $\alpha$ with $|\alpha|\leq 3$. The function $\eta_2(t)$ is chosen so that 
$\eta_2\in C_0^{\infty}(t_0-\rho^2, t_0+\rho^2)$, $0\leq\eta_2\leq 1$ in $\RR$, $\eta_2(t)\equiv 1$ for $t\in [t_0-s^2, t_0+s^2]$, and $$|\eta'_2(t)|\leq \frac{c}{\rho^2-s^2}\leq \frac{c}{r(\rho-s)}.$$
Then it holds that 
$$|\phi_t|\leq \frac{c}{r(\rho-s)},  \qquad|\nabla \phi_t|\leq \frac{c}{r(\rho-s)^2},$$
$$ |\nabla^3 \phi|\leq  \frac{c}{(\rho-s)^3},\qquad |\nabla^2 \phi|\leq  \frac{c}{(\rho-s)^2},\qquad |\nabla  \phi|\leq  \frac{c}{\rho-s}.$$

We next define
$$I(s)=I_1(s) +I_2(s),$$
where
$$I_1(s)=\sup_{t_0-s^2\leq t\leq t_0}\int_{B_s(x_0)}|u(x, t)|^2 dx=s\, A(z_0,s)$$
and 
$$I_2(s)=\int_{t_0-s^2}^{t_0}\int_{B_s(x_0)}|\nabla u(x, t)|^2 dx dt=s\, B(z_0,s).$$

For $0\leq \sigma \leq 1$, using $\phi$ as a test function in the generalized energy inequality we find
\begin{eqnarray}\label{Isenergy}
\lefteqn{I(s)}\nonumber\\
&\leq& \int_{t_0-\rho^2}^{t_0} \norm{|u|^2}_{L^{-\sigma, \, 2}(B_\rho(x_0))}\norm{\phi_t + \Delta \phi}_{L^{\sigma, \, 2}(B_{\rho}(x_0))} dt +\\
&& + \int_{t_0-\rho^2}^{t_0} \Big\{\norm{|u|^2 + 2p}_{L^{-\sigma, \, 2}(B_\rho(x_0))}\norm{ u\cdot \nabla \phi }_{L^{\sigma, \, 2}(B_{\rho}(x_0))}\Big\} dt\nonumber \\
&=:& J_1+ J_2.\nonumber
\end{eqnarray}

Applying the Gagliardo-Nirenberg type inequality (Lemma \ref{GNIne}), properties of the test function $\phi$, and H\"older's inequality 
we have 
\begin{eqnarray*}
J_1 &\leq& C \ds\int_{t_0-\rho^2}^{t_0} \norm{|u|^2}_{L^{-\sigma, \, 2}(B_\rho(x_0))}\times \nonumber\\
&&\quad\times\left\{\norm{\phi_t + \Delta \phi}_{L^{2}(B_{\rho}(x_0))}^{1-\sigma}\norm{\nabla\phi_t +\nabla \Delta \phi}_{L^{2}(B_{\rho}(x_0))}^{\sigma}\right\}dt\nonumber \\
&\leq&C \dfrac{\rho^{\frac {3}{2}}}{(\rho-s)^{2+\sigma}} \ds \int_{t_0-\rho^2}^{t_0} \norm{|u|^2}_{L^{-\sigma, \, 2}(B_\rho(x_0))} dt\\
&\leq&C \dfrac{\rho^{\frac {3}{2}+\sigma}}{(\rho-s)^{2+\sigma}}
\left( \ds\int_{t_0-\rho^2}^{t_0} \norm{|u|^2}_{L^{-\sigma,2}(B_\rho(x_0))}^{\frac{2}{2-\sigma}} dt\right)^{\frac{2-\sigma}{2}} .
 \end{eqnarray*}

Similarly,
\begin{eqnarray*}
J_2 &\leq&C \left(
\ds\int_{t_0-\rho^2}^{t_0}\norm{|u|^2 + 2p}^{\frac{2}{2-\sigma}}_{L^{-\sigma,2}(B_\rho(x_0))}dt\right)^{\frac{2-\sigma}{2}}\times \\
&\times&\left(\ds\int_{t_0-\rho^2}^{t_0}\norm{ u\cdot \nabla \phi }^{\frac{2(1-\sigma)}{\sigma}}_{L^{2}(B_{\rho}(x_0))} 
 \norm{\nabla u\cdot \nabla \phi +u \cdot \nabla^2 \phi}^{2}_{L^{2}(B_{\rho}(x_0))} dt\right)^{\frac{\sigma}{2}}.
 \end{eqnarray*}

Let us set 
$$X= \int_{t_0-r^2}^{t_0} \norm{|u|^2}_{L^{-\sigma,2}(B_r(x_0))}^{\frac{2}{2-\sigma}} dt,$$  
 $$Y=\int_{t_0-r^2}^{t_0} \norm{|u|^2 + 2p}_{L^{-\sigma,2}(B_r(x_0))}^{\frac{2}{2-\sigma}} dt.$$

Then combining  \eqref{Isenergy} with the estimates for $J_1$ and $J_2$,  it follows that  
\begin{eqnarray*}
I(s)&\leq&  C \frac{\rho^{\frac {3}{2}+\sigma}}{(\rho-s)^{2+\sigma}
}X^\frac {2-\sigma}{2}+ Y^{\frac {2-\sigma}{2}}\sup_{t_0-\rho^2 \leq t\leq t_0}\norm{u\cdot \nabla \phi}_{L^{2}(B_{\rho}(x_0))}^{1-\sigma} \times\nonumber\\
&&\qquad\times \left(\int_{t_0-\rho^2}^{t_0}
 \norm{\nabla u\cdot \nabla \phi +u \cdot \nabla^2 \phi}^{2}_{L^{2}(B_{\rho}(x_0))} dt\right)^{ \frac{\sigma}{2}}\nonumber\\
&\leq&  C \frac{\rho^{\frac {3}{2}+\sigma}}{(\rho-s)^{2+\sigma}}X^\frac {2-\sigma}{2}
+ C\,Y^{\frac {2-\sigma}{2}} \left(\frac{I_1(\rho)}{(\rho-s)^2}\right)^{\frac{1-\sigma}{2}}\times\nonumber\\
&&\qquad\times  \left\{\left(\frac{\rho^2 I_1(\rho)}{(\rho-s)^4}\right)^{ \frac\sigma2} +
\left(\frac{ I_2(\rho)}{(\rho-s)^2}\right)^{\frac\sigma2} \right\}.\nonumber
\end{eqnarray*}

Thus, using $r/2\leq \rho\leq r$ and $I_1(\rho), I_2(\rho)\leq I(\rho)$, we get
\begin{eqnarray}
I(s) &\leq&  C \frac{\rho^{\frac {3}{2}+\sigma}}{(\rho-s)^{2+\sigma}}X^{\frac {2-\sigma}{2}}  
+ C\, Y^{\frac {2-\sigma}{2}} \frac{ \rho^{\sigma} I_1(\rho)^{\frac{1}{2}}}{(\rho-s)^{1+\sigma}}\nonumber\\
&&+\, C\, Y^{\frac {2-\sigma}{2}} \left(\frac{I_1(\rho)}{(\rho-s)^2}\right)^{\frac{1-\sigma}{2}} 
\left(\frac{ I_2(\rho)}{(\rho-s)^2}\right)^{\frac{\sigma}{2}} \nonumber\\
&\leq&    C \frac{r^{\frac {3}{2}+\sigma}}{(\rho-s)^{2+\sigma}}X^{\frac {2-\sigma}{2}}    
+ C\,  \frac{Y^{\frac {2-\sigma}{2}}  r^{\sigma} }{(\rho-s)^{1+\sigma}} I(\rho)^{\frac12} + C\,  \frac{Y^{\frac {2-\sigma}{2}} }{(\rho-s)} I(\rho)^{\frac12}.\nonumber
\end{eqnarray}

Then by Young's inequality it follows that 
\begin{eqnarray*}
I(s) &\leq & C \frac{r^{\frac {3}{2}+\sigma}}{(\rho-s)^{2+\sigma}}X^{\frac {2-\sigma}{2}} 
  + C\,\left[  \frac{ r^{2\sigma}}{(\rho-s)^{2+2\sigma}} +  \frac{1}{(\rho-s)^2} \right] Y^{ {2-\sigma}}\\
	&&+\, \frac{1}{2}I(\rho). 
\end{eqnarray*}

As this holds for all $r/2\leq s<\rho\leq r$ by Lemma \ref{Giusti-lem} we obtain
$$I(r/2) \leq C\, \frac{X^{\frac {2-\sigma}{2}}}{r^{1/2}} + C\, \frac{Y^{2-\sigma}}{r^2},$$
from which the proposition follows.

\end{proof}

By   Lemma \ref{bound-sigma} we have the following consequence of Proposition \ref{ABcontrol0}.

\begin{proposition} \label{ABcontrol1b} Suppose that $(u,p)$ is a suitable weak solution 
to the Navier-Stokes equations in  $Q_r(z_0)$. Then one has
\begin{equation*}
A(z_0, r/2) + B(z_0, r/2)\leq C[ C_{\alpha,\beta}(z_0, r)^{\frac{1}{\beta}} +  C_{\alpha,\beta}(z_0, r)^{\frac{2}{\beta}} +  D_{\alpha,\beta}(z_0, r)^{\frac{2}{\beta}}]
\end{equation*}
for any $\alpha\in [6/5,2]$ and $\beta=\frac{4\alpha}{7\alpha-6}$.
\end{proposition}

\section{$\epsilon$-regularity criteria}

In this section we prove Theorems \ref{VassThmp=1} and \ref{PhThm-sigma}. We start with the following lemma.
\begin{lemma}\label{boundH}
Let  $h$ be a harmonic function in $B_{2r}(x_0)$ and $0\leq \sigma\leq 1$. Then we have 
\begin{eqnarray*}
\norm{h}_{\Lp{2}{B_r(x_0)}}
&\leq&   C \norm{\frac{h}{r^{\sigma}}}_{\Lp{-\sigma,2}{B_{2r}(x_0)}}.
\end{eqnarray*}
\end{lemma}
\begin{proof}
The case $\sigma=0$ is obvious. We thus assume that  $0<\sigma\leq 1$.
Let $f$ be a harmonic function in $B_2(0)$. Let $\varphi \in \Ci{B_{3/2}(0)}$ be such that $0\leq \varphi \leq 1$, $\varphi \equiv 1$ in $B_1(0)$ and $|\nabla \varphi|\leq c$. Hence, $f\varphi \in  \Ci{B_{3/2}(0)}$ and
\begin{align}\label{hvar}
\norm{f\varphi}_{\Lp{2}{B_{3/2}(0)}}&\leq
\norm{f\varphi}_{\Lp{-\sigma,2}{B_{3/2}(0)}}^{\frac 12}
\norm{f\varphi}_{\Lp{\sigma,2}{B_{3/2}(0)}}^{\frac 12}.
\end{align}

Observe that, for any $g \in \Lp{\sigma,2}{B_{3/2}(0)}$, by \cite[Theorem A.12]{KPV} we have 
$$\norm{g\varphi}_{\Lp{\sigma,2}{B_{3/2}(0)}}\leq C \norm{g}_{\Lp{\sigma,2}{B_{3/2}(0)}},$$
and thus
\begin{eqnarray*}
 \int_{B_{3/2}(0)} f  \varphi  g dx &\leq&\norm{f}_{\Lp{-\sigma,2}{B_{3/2}(0)}}\norm{g\varphi}_{\Lp{\sigma,2}{B_{3/2}(0)}} \\
&\leq& C \norm{f}_{\Lp{-\sigma,2}{B_{3/2}(0)}}\norm{g}_{\Lp{\sigma,2}{B_{3/2}(0)}}.
\end{eqnarray*}

 This means that 
\begin{equation} \label{est:hI}
\norm{f\varphi}_{\Lp{-\sigma,2}{B_{3/2}(0)}}\leq C \norm{f}_{\Lp{-\sigma,2}{B_{3/2}(0)}}.  
\end{equation}


Also,
\begin{eqnarray}
\norm{f\varphi}_{\Lp{\sigma,2}{B_{3/2}(0)}}&\leq&
\norm{f\varphi}^{1-\sigma}_{\Lp{2}{B_{3/2}(0)}}\norm{\nabla(f\varphi)}^{\sigma}_{\Lp{2}{B_{3/2}(0)}}\nonumber\\
&\leq&
\norm{f\varphi}^{1-\sigma}_{\Lp{2}{B_{3/2}(0)}}\norm{(\nabla f) \varphi +  f(\nabla \varphi)}^{\sigma}_{\Lp{2}{B_{3/2}(0)}}\nonumber\\
&\leq&c
\norm{f}^{1-\sigma}_{\Lp{2}{B_{3/2}(0)}}\left(\norm{\nabla f}^{\sigma}_{\Lp{2}{B_{3/2}(x_0)}}+
\norm{f}^{\sigma}_{\Lp{2}{B_{3/2}(0)}}\right)\nonumber\\
&\leq&c
\norm{f}^{1-\sigma}_{\Lp{2}{B_{3/2}(0)}}\left(\norm{f}^{\sigma}_{\Lp{2}{B_2(0)}}+
\norm{f}^{\sigma}_{\Lp{2}{B_{3/2}(0)}}\right)\nonumber\\
&\leq&c \norm{f}_{\Lp{2}{B_2(0)}}\label{est:hII}.
\end{eqnarray}
Here in the 4th inequality we used the fact that $f$ is harmonic in $B_2(0)$.

Hence, \eqref{hvar}, \eqref{est:hI} and \eqref{est:hII} yield
\begin{align} \label{h-bound-1}
\norm{f}_{\Lp{2}{B_1(0)}}\leq \norm{f\varphi}_{\Lp{2}{B_{3/2}(0)}}& \leq C \norm{f}_{\Lp{-\sigma,2}{B_{3/2}(0)}}^{\frac 12}
\norm{f}^{\frac{1}2}_{\Lp{2}{B_2(0)}}.
\end{align}

Now for $r>0$, let   $h$ be a harmonic function in $B_{2r}(x_0)$. 
We define $f(x)=h(rx+x_0)$ for $x\in B_2(0)$.   Then $f$ is harmonic in $B_{2}(0)$. 

Note that for any $\varphi\in \Lp{\sigma,2}{B_{3/2}(0)}$ we have 
\begin{align*}
 &\norm{\varphi\left(\frac{\cdot-x_0}r\right)}_{\Lp{\sigma,2}{B_{3r/2}(x_0)}}
 =\left(\int_{\RR^3}|\xi|^{2\sigma}\left| \widehat{\varphi\left(\frac{\cdot-x_0}r\right) }\right|^2d\xi\right)^{\frac12}\\
 &\qquad = r^{3} \left(\int_{\RR^3}|\xi|^{2\sigma}\left| \widehat{\varphi}\left(r\xi\right) \right|^2d\xi\right)^{\frac12}\\
  &\qquad= r^{3} \left(\int_{\RR^3}\left|\frac{\zeta}{r}\right|^{2\sigma}\left| \widehat{\varphi}\left(\zeta \right) \right|^2 r^{-3}d\zeta \right)^{\frac12}
  =r^{\frac32-\sigma} \norm{\varphi}_{\Lp{\sigma,2}{B_{3/2}(0)}}.
  \end{align*}
Thus for such $\varphi$,
\begin{eqnarray*}
\lefteqn{\int_{B_{3/2}(0)} h(rx +x_0)\varphi(x)dx=r^{-3} \int_{B_{3r/2}(x_0)} h(y)\varphi \left(\frac{y-x_0}{r}\right)dy}\\
&\leq& r^{-3}\norm{h}_{\Lp{-\sigma,2}{B_{3r/2}(x_0)}} \norm{\varphi\left(\frac{\cdot-x_0}r\right)}_{\Lp{\sigma,2}{B_{3r/2}(x_0)}}\\
&\leq& r^{-\frac{3}{2}-\sigma}\norm{h}_{\Lp{-\sigma,2}{B_{3r/2}(x_0)}} \norm{\varphi}_{\Lp{\sigma,2}{B_{3/2}(0)}}.
\end{eqnarray*}

This implies that 
\begin{eqnarray*}
\norm{f}_{\Lp{-\sigma,2}{B_{3/2}(0)}}\leq  r^{-\frac{3}{2}}\norm{\frac{h}{r^{\sigma}}}_{\Lp{-\sigma,2}{B_{3r/2}(x_0)}},
\end{eqnarray*}
and by substituting into  \eqref{h-bound-1} we have
\begin{eqnarray*}
\left(\fint_{B_r(x_0)} |h|^2 dx\right)^\frac12
\leq C r^{-\frac{3}{4}}\norm{\frac{h}{r^{\sigma}}}_{\Lp{-\sigma,2}{B_{2r}(x_0)}}^{\frac12} 
  \left(\fint_{B_{2r}(x_0)} |h|^2 dx\right)^\frac14.
\end{eqnarray*}

Or equivalently, 
\begin{equation}\label{1and4}
\int_{B_r(x_0)} |h|^2 dx \leq C \norm{\frac{h}{r^{\sigma}}}_{\Lp{-\sigma,2}{B_{2r}(x_0)}}  \left(\int_{B_{2r}(x_0)} |h|^2 dx\right)^\frac12.
\end{equation}

Let $r\leq s<t \leq 2r$.  The ball $B_s(x_0)$ can be covered by a collection of balls $\left\{B_i=B_{\frac{t-s}{2}}(x_i):\; x_i\in B_s(x_0)\right\}$, in such a way that  each point $y\in \RR^n$ belongs to at most $N=N(n)$ balls in the collection  $\left\{2B_i=B_{t-s}(x_i)\right\}$, that is, 
$$\sum_i \chi_{_{2B_i}}(y)\leq N(n).$$
Then applying \eqref{1and4} to the balls $B_i$, we find
\begin{eqnarray*}
\int_{B_i} |h|^2 dx
&\leq&  C\norm{\frac{h}{(t-s)^{\sigma}}}_{\Lp{-\sigma,2}{2B_{i}}} \left(\int_{2B_{i}} |h|^2 dx\right)^\frac12\\
&\leq&  C\norm{\frac {h}{(t-s)^{\sigma}}}_{\Lp{-\sigma,2}{B_{2r}(x_0)}} \left(\int_{2B_{i}} |h|^2 dx\right)^\frac12.
\end{eqnarray*}
Thus,
\begin{eqnarray*}
\int_{B_s(x_0)} |h|^2 dx&\leq& \sum_i \int_{B_i} |h|^2 dx\\
&\leq&  C\norm{\frac {h}{(t-s)^{\sigma}}}_{\Lp{-\sigma,2}{B_{2r}(x_0)}}  \sum_i \left(\int_{2B_{i}} |h|^2 dx\right)^\frac12.
\end{eqnarray*}
Note that 
\begin{eqnarray*}
 \sum_i\int_{2B_{i}} |h|^2 dx= \int_{B_{t}(x_0)} |h|^2 \sum_i \chi_{_{2B_i}}(x) dx\leq  N(n)\int_{B_{t}(x_0)} |h|^2 dx,
\end{eqnarray*}
and thus 
\begin{eqnarray*}
\int_{B_s(x_0)} |h|^2 dx \leq  C(n)\norm{\frac {h}{(t-s)^{\sigma}}}_{\Lp{-\sigma,2}{B_{2r}(x_0)}} \left(\int_{B_{t}(x_0)} |h|^2 dx\right)^\frac12.
\end{eqnarray*}
Then by Young's inequality it follows that 
\begin{eqnarray*}
\int_{B_s(x_0)} |h|^2 dx \leq  C (t-s)^{-2\sigma}  \norm{h}_{\Lp{-\sigma,2}{B_{2r}(x_0)}}^2 + \frac 1 2 \int_{B_{t}(x_0)} |h|^2 dx.
\end{eqnarray*}
Thus applying  Lemma \ref{Giusti-lem} we have 
\begin{eqnarray*}
\int_{B_r(x_0)} |h|^2 dx
&\leq& C r^{-2\sigma} \norm{h}_{\Lp{-\sigma,2}{B_{2r}(x_0)}}^2 
\end{eqnarray*}
as desired.
\end{proof} 

The next lemma provides bounds for the pressure.
\begin{lemma}\label{Dbar}
Suppose that $(u,p)$ is a suitable weak solution 
to the Navier-Stokes equations in  $Q_{\rho}(z_0)$. For any $r\in (0, \rho/2]$ we have the following bounds:
\begin{eqnarray}\label{p}
\lefteqn{r^{-\frac 3 2} \int_{t_0-r^2}^{t_0} \norm{p(\cdot, t)}_{L^2(B_r(x_0))} dt}\nonumber\\
&\leq& C  \rho^{-3}\int_{t_0-\rho^2}^{t_0} \norm{p}_{L^{1}(B_\rho(x_0))} dt+   C  \Big(\frac{\rho}{r}\Big)^{3/2} A(z_0,\rho)^{1/4} B(z_0,\rho)^{3/4},
\end{eqnarray}
\begin{eqnarray}\label{pbar}
\lefteqn{r^{-\frac 3 2} \int_{t_0-r^2}^{t_0} \norm{p(\cdot, t)- [p(\cdot, t)]_{x_0,r}}_{L^2(B_r(x_0))} dt}\nonumber\\
 &\leq& C\, \Big(\frac{r}{\rho}\Big) \, \rho^{-3} \int_{t_0-\rho^2}^{t_0} \norm{p(\cdot, t)- [p(\cdot, t)]_{x_0,\rho}}_{L^1(B_\rho(x_0))} dt + \\
&& +\,  C  \Big(\frac{\rho}{r}\Big)^{3/2} A(z_0,\rho)^{1/4} B(z_0,\rho)^{3/4},\nonumber
\end{eqnarray}
and
\begin{eqnarray}\label{DDtiti}
\lefteqn{r^{-3}\int_{t_0-r^2}^{t_0} \norm{p}_{L^{1}(B_r(x_0))} dt} \nonumber\\
&\leq& C  \rho^{-\frac3 2 -\sigma} \int_{t_0-\rho^2}^{t_0}\norm{p}_{L^{-\sigma,2}(B_\rho(x_0))} dt +\\
&& +\,  C  \Big(\frac{\rho}{r}\Big)^{3} A(z_0,\rho)^{1/4+\sigma/2} B(z_0,\rho)^{3/4-\sigma/2}\nonumber
\end{eqnarray}
for any $\sigma \in \left[0,3/2\right)$.
\end{lemma}

\begin{proof} Let $h_{x_0, \rho}=h_{x_0, \rho}(\cdot, t)$ be a function on $B_\rho(x_0)$ for a.e. $t$ such that
$$h_{x_0, \rho}=p-\tilde{p}_{x_0, \rho}\quad {\rm in~}  B_\rho(x_0),$$ 
where $\tilde{p}_{x_0, \rho}$ is defined by 
$$\tilde{p}_{x_0, \rho}=R_iR_j[(u_i-[u_i]_{x_0,\rho})(u_j-[u_j]_{x_0,\rho})\chi_{B_\rho(x_0)}].$$
Here $R_i=D_i(-\Delta)^{-\frac{1}{2}}$, $i=1,2,3$, is the $i$-th Riesz transform,
and we  used the notation 
$$[f]_{x_0, \rho}:=\fint_{B_\rho(x_0)} f(x) \, dx=\frac{1}{|B_\rho(x_0)|}\int_{B_\rho(x_0)} f(x) \, dx.$$
to denote the spatial average of a function $f$ over the ball $B_\rho(x_0)$.

 Note that for any $\varphi\in C_0^{\infty}(B_\rho(x_0))$, we have
\begin{eqnarray*}
-\int_{B_\rho(x_0)} \tilde{p}_{x_0, \rho}\Delta \varphi dx&=&\int_{B_{\rho}(x_0)}(u_i-[u_i]_{x_0,\rho})(u_j-[u_j]_{x_0,\rho}) D_{ij}\varphi\, dx\\
&=&\int_{B}u_i u_j D_{ij}\varphi\, dx,
\end{eqnarray*}
which follows from  the properties $-R_i R_j(\Delta \varphi)=D_{ij}\varphi$ and ${\rm div}\, u=0$.
Thus, as $p$ also solves
$$-\Delta p={\rm div}\, {\rm div} (u\otimes u)$$
in the  distributional sense, we see that  $h_{x_0, \rho}$ is harmonic in  $B_{\rho}(x_0)$ for a.e. $t$.
Then   for $r\in (0, \rho/2]$ it holds that
\begin{eqnarray*}
 \left(\fint_{B_{r}(x_0)} |h_{x_0, \rho}|^{2} dx\right)^{\frac 1 2} \leq  C \fint_{B_{\rho}(x_0)} |h_{x_0, \rho}| dx
\end{eqnarray*}
and
\begin{eqnarray*}
 \left(\fint_{B_{r}(x_0)} |h_{x_0, \rho}- [h_{x_0, \rho}]_{x_0,r}|^{2} dx\right)^{\frac 1 2}&\leq& C \frac{r}{\rho} \fint_{B_{\rho}(x_0)} |h_{x_0, \rho}- [h_{x_0, \rho}]_{x_0,\rho}| dx.
\end{eqnarray*}

Then using $p=\tilde{p}_{x_0,\rho} +h_{x_0,\rho}$, they give
\begin{eqnarray}\label{NOAVG}
\lefteqn{\int_{B_r(x_0)} |p(x,t)|^2 dx}\nonumber\\
 &\leq& 2\int_{B_\rho(x_0)} |\tilde{p}_{x_0, \rho}|^2 dx  +  C \frac{r^3}{\rho^{6}}\norm{{h_{x_0, \rho}} }_{L^1(B_{\rho}(x_0))}^2,
\end{eqnarray}
and
\begin{eqnarray}\label{WAVG}
\lefteqn{\int_{B_r(x_0)} |p(x,t)-[p(\cdot,t)]_{x_0,r}|^2 dx}\nonumber\\
&\leq& 2\int_{B_\rho(x_0)} |\tilde{p}_{x_0, \rho}|^2 dx  +  C \frac{r^5}{\rho^{8}}\norm{{h_{x_0, \rho}} - [h_{x_0,\rho}]_{x_0,\rho}}_{L^1(B_{\rho}(x_0))}^2.
\end{eqnarray}

Now by \eqref{NOAVG} and $h_{x_0,\rho}=p-\tilde{p}_{x_0,\rho} $ we obtain
\begin{eqnarray}\label{pptil}
\int_{B_r(x_0)} |p(x,t)|^2 dx &\leq&  2\int_{B_\rho(x_0)} |\tilde{p}_{x_0, \rho}|^2 dx+\nonumber\\
&&+\,  C \frac{r^3}{\rho^{6}}\left(\norm{  {\tilde p_{x_0, \rho}} }_{L^1(B_{\rho}(x_0))}^2+\norm{  p }_{L^1(B_{\rho}(x_0))}^2 \right)\nonumber\\
&\leq& C \int_{B_\rho(x_0)} |\tilde{p}_{x_0, \rho}|^2 dx  +  C \frac{r^3}{\rho^{6}}\norm{ p }_{L^1(B_{\rho}(x_0))}^2,
\end{eqnarray}
where  we used H\"older's inequality and the fact that $r/\rho\leq 1/2$.

On the other hand, by the Calder\'on-Zygmund estimate and Lemma \ref{Sob-inte}   we find
\begin{eqnarray}\label{ptil1}
\lefteqn{\int_{B_\rho(x_0)}|\tilde{p}_{x_0,\rho}|^{2} dx\leq C \int_{B_\rho(x_0)}| u-[u]_{x_0,\rho}|^{4} dx}\nonumber\\
&\leq& C \Big(\int_{B_\rho(x_0)}| \nabla u|^{2} dx\Big)^{3/2} \Big(\int_{B_\rho(x_0)}|u|^{2} dx\Big)^{1/2}.
\end{eqnarray}

Combining \eqref{pptil} and \eqref{ptil1}  we have 
\begin{eqnarray*}
\lefteqn{\norm{ p }_{L^2(B_{r}(x_0))} \leq  C \frac{r^{3/2}}{\rho^{3}}\norm{ p }_{L^1(B_{\rho}(x_0))}+}\\
&& +\, C \Big(\int_{B_\rho(x_0)}| \nabla u|^{2} dx\Big)^{3/4} \Big(\int_{B_\rho(x_0)}|u|^{2} dx\Big)^{1/4}.
\end{eqnarray*}

Integrating the last bound with respect to $r^{-3/2}dt$ over the interval $(t_0-r^2, t_0)$ and using H\"older's inequality  we obtain inequality
\eqref{p}. 

Likewise, using \eqref{WAVG}  instead of \eqref{NOAVG} and arguing similarly we obtain inequality \eqref{pbar}. We remark that in this case we also need to use the elementary fact that 
$$\norm{{h_{x_0, \rho}} - [h_{x_0,\rho}]_{x_0,\rho}}_{L^1(B_{\rho}(x_0))}\leq 2 \norm{{h_{x_0, \rho}} - [p(\cdot,t)_{x_0,\rho}]_{x_0,\rho}}_{L^1(B_{\rho}(x_0))}.$$

As for \eqref{DDtiti}, we first bound
\begin{eqnarray*}
\int_{B_r(x_0)} |p(x,t)|dx &\leq&  \int_{B_\rho(x_0)} |\tilde{p}_{x_0, \rho}| dx + \int_{B_r(x_0)} |h_{x_0, \rho}|dx\\
&\leq&\int_{B_\rho(x_0)} |\tilde{p}_{x_0, \rho}| dx + C r^3 \left(\fint_{B_{\rho/2}(x_0)} |h_{x_0, \rho}|^2dx \right)^{1/2}\\
&\leq& \int_{B_\rho(x_0)} |\tilde{p}_{x_0, \rho}| dx  +  C \frac{r^3}{\rho^{\frac{3}{2}+\sigma}}\norm{  {h_{x_0, \rho}} }_{\Lp{-\sigma,2}{B_{\rho}(x_0)}}\\
&\leq&  \int_{B_\rho(x_0)} |\tilde{p}_{x_0, \rho}| dx+\\
&+& C \frac{r^3}{\rho^{\frac{3}{2} +\sigma}}\left(\norm{  {\tilde p_{x_0, \rho}} }_{\Lp{-\sigma,2}{B_{\rho}(x_0)}}+\norm{  p }_{\Lp{-\sigma,2}{B_{\rho}(x_0)}} \right).
\end{eqnarray*}
Here we used Lemma \ref{boundH} in the third inequality.

Now using H\"older's inequality, Lemma \ref{bound-sigma} with $\sigma\in[0,3/2)$, and $r/\rho\leq 1/2$,  we have
\begin{eqnarray*}
\int_{B_\rho(x_0)} |\tilde{p}_{x_0, \rho}| dx&+& C \frac{r^3}{\rho^{\frac{3}{2}+\sigma}} \norm{ {\tilde{p}_{x_0, \rho}} }_{\Lp{-\sigma,2}{B_{\rho}(x_0)}}\\
 & \leq& C \rho^{3/2-\sigma} \norm{ {\tilde{p}_{x_0, \rho}} }_{\Lp{\frac{6}{3+2\sigma}}{B_{\rho}(x_0)}}.
\end{eqnarray*}

Thus,
\begin{eqnarray}\label{pptil2}
\int_{B_r(x_0)} |p(x,t)| dx &\leq& C \rho^{3/2-\sigma} \norm{ {\tilde{p}_{x_0, \rho}} }_{\Lp{\frac{6}{3+2\sigma}}{B_{\rho}(x_0)}}\\
&& +\,  C \frac{r^3}{\rho^{\frac{3}{2}+\sigma}}\norm{ p }_{\Lp{-\sigma,2}{B_{\rho}(x_0)}}\nonumber.
\end{eqnarray}

As before, the $L^{\frac{6}{3+2\sigma}}$ norm of $\tilde{p}_{x_0, \rho}$ is treated using Calder\'on-Zygmund estimate and Lemma \ref{Sob-inte} which give
\begin{eqnarray}\label{ptil12}
\lefteqn{\norm{ {\tilde{p}_{x_0, \rho}} }_{\Lp{\frac{6}{3+2\sigma}}{B_{\rho}(x_0)}}} \nonumber\\
&\leq& C \Big(\int_{B_\rho(x_0)}| \nabla u|^{2} dx\Big)^{3/4-\sigma/2} \Big(\int_{B_\rho(x_0)}|u|^{2} dx\Big)^{1/4+\sigma/2}.
\end{eqnarray}

Combining \eqref{pptil2}, \eqref{ptil12}  we have 
\begin{eqnarray*}
\lefteqn{\int_{B_r(x_0)} |p(x,t)| dx \leq  C \frac{r^3}{\rho^{3/2+\sigma}}\norm{ p }_{\Lp{-\sigma,2}{B_{\rho}(x_0)}}+}\\
&& +\, C \rho^{3/2-\sigma}\Big(\int_{B_\rho(x_0)}| \nabla u|^{2} dx\Big)^{3/4-\sigma/2} \Big(\int_{B_\rho(x_0)}|u|^{2} dx\Big)^{1/4+\sigma/2},
\end{eqnarray*}
from which integrating in $t$ we obtain \eqref{DDtiti}.
\end{proof}

We now recall  the following $\epsilon$-regularity criterion for suitable weak solutions to the Navier-Stokes equations (see \cite[Lemma 3.3]{SS}).
\begin{lemma}\label{SS-02} There exists a positive number $\epsilon_{\star}$ such that the following property holds.  If $(u,p)$ be a suitable solution to the Navier-Stokes equations in $Q_{R_\star}(z_0)$ for some $R_\star>0$ such that 
$$\sup_{0<r < R_\star}A(z_0,r)\leq \epsilon_\star,$$ 
then  $z_0$ is a regular point of $u$. 
\end{lemma}

We are now ready to prove Theorem \ref{VassThmp=1}.

\begin{proof}[Proof of Theorem \ref{VassThmp=1}] Our assumption is that 
$$A((0,0),1) + B((0,0),1)+ \int_{-1}^{0} \norm{p}_{L^1(B_1(0))}dt\leq \epsilon,$$
where $\epsilon\in(0,1)$ is to be determined.
By Lemma \ref{SS-02}, it is enough to show that 
$$\sup_{0<r<1/4} [A(z,r) + B(z,r)]\leq C\epsilon^{\frac12} \leq \epsilon_\star.$$
for every $z \in Q_{1/2}(0,0)$.  Here $C$ is independent of $r$ and $z$.
 By translation invariance, it suffices to consider the case $z=0$. 
 Moreover, it suffices to show a discrete version, i.e.,  we just need to show that 
\begin{equation}\label{rn}
A((0,0),\theta^n) + B((0,0),\theta^n)\leq  \epsilon^{\frac12}
\end{equation}
for a fixed $\theta\in(0,1/4]$ and for all $n=1,2,\dots$  The discretization enables us to use an inductive argument  in the spirit of \cite[Section 4]{CKN} and \cite{WZ}. 

Let $\theta\in(0,1/4]$ be determined later and define $$r_n=\theta^n , \quad n\in \NN.$$
  By our hypothesis, inequality \eqref{rn} holds in the case $n=1$ provided $\epsilon_0$ is sufficiently small (depending on $\theta$). Suppose now that it holds for $n=1,\dots, m-1$ with an $m\geq 2$. Let 
$\phi_m=\chi\psi_m$, where $0\leq \chi\leq 1$ is a smooth cutoff function which equals 1 on $Q_{\theta^2}(0,0)$ and vanishes in  $\RR^3\times(-\infty,0)\setminus Q_{2\theta/3}(0,0)$, and $\psi_m$ is given by
$$\psi_m(x,t)=(r_m^2-t)^{-3/2} e^{-\frac{|x|^2}{4(r_m^2-t)}}, \qquad\qquad t<r_m^2.$$ 
Then it can be seen that  $\phi_m\geq 0$, $(\partial_t +\Delta) \phi_m =0$ in $Q_{\theta^2}(0,0)$, and 
$$|(\partial_t +\Delta) \phi_m|\leq C \quad {\rm on~} Q_{2\theta/3}(0,0),$$
$$2^{-3/2}\, r_m^{-3}\leq \phi_m\leq  r_m^{-3}, \quad |\nabla \phi_m|\leq C r_m^{-4}  {\rm ~~on~~}  Q_{r_m}(0,0), \qquad m\geq 2,$$
$$\phi_m\leq C r_k^{-3}, \quad |\nabla \phi_m|\leq C r_k^{-4} {\rm ~~on~~}  Q_{r_{k-1}}(0,0)\setminus Q_{r_{k}}(0,0), \qquad 1<k\leq m.$$
Here the constant $C=C(\theta)$ is independent of $m$. 

Using $\phi_m$ as a test function in the generalized energy inequality, we find that 
\begin{equation}\label{ABCIII}
A((0,0), r_m) + B((0,0), r_m)\leq C (I + II+III),  
\end{equation}
where 
$$I=r_m^2 \int_{Q_{\theta}(0,0)} |u|^2 dxdt,$$
$$II=r_m^2\int_{Q_{\theta}(0,0)} |u|^3 |\nabla \phi_m| dx dt,$$
$$III=r_m^2 \left|\int_{Q_{\theta}(0,0)} p(u\cdot\nabla \phi_m) dx dt\right|.$$

By the hypothesis, we have 
$$I\leq  r_m^2 \epsilon\leq  \epsilon^{\frac34}.$$

By the above properties of $\phi_m$, we have
\begin{eqnarray*} 
 II&=& r_m^2 \sum_{k=1}^{m-1} \int_{Q_{r_k}\setminus Q_{r_{k+1}}} |u|^3 |\nabla \phi_m| dx dt + r^2_m\int_{Q_{r_m}} |u|^3 |\nabla \phi_m| dx dt\\
 &\leq& C r_m^2 \sum_{k=1}^{m-1} r_k^{-4} \int_{Q_{r_k}} |u|^3 dxdt.
\end{eqnarray*}
Thus by Lemma \ref{boundC1} and inductive hypothesis, it follows that 
\begin{equation*} 
 II \leq C r_m^2 \sum_{k=1}^{m-1} r_k^{-2} \epsilon^{3/4} \leq C \epsilon^{3/4}.
\end{equation*}

As for the term $III$, we write
$$\phi_m=\chi_1\phi_m=\sum_{k=1}^{m-1} (\chi_k-\chi_{k+1}) \phi_m + \chi_m\phi_m,$$
where $\chi_k$, $k=1,2,\dots,m$, is a smooth cutoff function  such that 
$0\leq\chi_k\leq 1$, $\chi_k=1$ in $Q_{7r_k/8}(0,0)$, $\chi_k=0$ in $\RR^3\times(-\infty,0)\setminus Q_{r_k}(0,0)$, and
$|\nabla \chi_k| \leq C/r_{k}$. Then
\begin{eqnarray*}
III&\leq& r_m^2\left|\sum_{k=1}^{m-1} \int_{Q_{r_k}}p u\cdot \nabla[(\chi_k-\chi_{k+1}) \phi_m] dxdt \right|\\
&& +\, r_m^2 \left| \int_{Q_{r_m}} p u\cdot\nabla(\chi_m \phi_m) dxdt \right|\\
&=& r_m^2\left|\sum_{k=1}^{m-1} \int_{Q_{r_k}} (p-[p]_{0,r_k}) u\cdot \nabla[(\chi_k-\chi_{k+1}) \phi_m] dxdt \right|\\
&& +\, r_m^2 \left| \int_{Q_{r_m}} (p-[p]_{0, r_m}) u\cdot\nabla(\chi_m \phi_m) dxdt \right|,
\end{eqnarray*}
where we used the fact that $u$ is divergence-free. Then by H\"older's inequality and the properties of $\phi_m$, we see that
\begin{eqnarray}
III &\leq& C r_m^2 \sum_{k=2}^{m}  r_k^{-4}  \int_{Q_{r_k}} |(p-[p]_{0,r_k}) u|  dxdt  \nonumber\\
&& +\, C\theta^{-2}   \int_{Q_{\theta}} |(p-[p]_{0,\theta}) u| dxdt\nonumber\\
&\leq& C r_m^2 \sum_{k=2}^{m}  r_k^{-4} \int_{-r_{k}^2}^0 \norm{p-[p]_{0,r_k}}_{L^2(B_{r_k})} \norm{u}_{L^2(B_{r_k})}  dxdt\nonumber\\
&& +\, C \theta^{-2} \int_{-\theta^2}^0 \norm{p-[p]_{0,\theta}}_{L^2(B_{\theta})} \norm{u}_{L^2(B_{\theta})}  dxdt.\nonumber
\end{eqnarray}

By inductive hypothesis, this gives 

\begin{eqnarray}\label{boundIII}
III &\leq& C r_m^2 \sum_{k=2}^{m}  r_k^{-2}\, \epsilon^{\frac14}\, r_k^{-3/2} \int_{-r_{k}^2}^0 \norm{p-[p]_{0,r_k}}_{L^2(B_{r_k})}  dxdt\\
&& +\, C \, \epsilon^{\frac12} \, \theta^{-3/2}\int_{-\theta^2}^0 \norm{p-[p]_{0,\theta}}_{L^2(B_{\theta})} dxdt.\nonumber
\end{eqnarray}
Here the constant $C$ could depend on $\theta$.

We now let $A(k)=A((0,0), r_{k})$, $B(k)=B((0,0), r_{k})$, and  
$$U(k)=r_k^{-3/2} \int_{-r_{k}^2}^0 \norm{p-[p]_{0,r_k}}_{L^2(B_{r_k})}  dxdt.$$
By Lemma \ref{Dbar} and H\"older's inequality,  for $2\leq k\leq m$ we have 
\begin{eqnarray*}
U(k)\leq (C\theta) U(k-1) + C\, \theta^{-3/2} A(k-1)^{1/4} B(k-1)^{3/4},  
\end{eqnarray*}
where $C\geq1$ is independent of $k$ and $\theta$. Choosing $\theta=\frac{1}{4C}$ and iterating this inequality we obtain
$$U(k)=(1/4)^{k-1} U(1) + C\, \theta^{-3/2} \sum_{\ell=1}^{k-1} (1/4)^{\ell-1} A(k-\ell)^{1/4} B(k-\ell)^{3/4}.$$

 Then by inductive hypothesis we find
\begin{eqnarray*}
U(k)&\leq&  U(1) + C \sum_{\ell=1}^{k-1} (1/4)^{\ell-1} \epsilon^{\frac12}\\
&\leq& \theta^{-3/2} \int_{-\theta^2}^0 \norm{p-[p]_{0,\theta}}_{L^2(B_{\theta})} dxdt + C\epsilon^{\frac12}.
\end{eqnarray*}

Combining this with \eqref{boundIII} we arrive at 
$$III\leq C\, \epsilon^{\frac14} \, \theta^{-3/2}\int_{-\theta^2}^0 \norm{p-[p]_{0,\theta}}_{L^2(B_{\theta})} dxdt + C\epsilon^{\frac34},$$
which by Lemma \ref{Dbar} gives 
\begin{eqnarray*}
III&\leq& C \epsilon^{\frac14} [D((0,0),2\theta) + A((0,0), 2\theta)^{1/4}B((0,0), 2\theta)^{3/4}] + C \epsilon^{\frac34}\\
&\leq& C (\epsilon^{\frac54} +\epsilon^{\frac34})\leq 2C\, \epsilon^{\frac34}.
\end{eqnarray*}

Combining \ref{ABCIII} and the estimates for $I$, $II$ and $III$ we obtain
$$A((0,0), r_m) + B((0,0), r_m) \leq C\, \epsilon^{\frac34}\leq \epsilon^\frac12$$
provided $\epsilon$ is small enough. This proves \eqref{rn} and the proof is complete.
\end{proof}

Using Lemma \ref{Dbar} and a covering argument we obtain the following consequence of Theorem \ref{VassThmp=1}.

\begin{corollary}\label{L1-sigma} Let $\sigma\in[0,1]$. There exists a  number $\epsilon\in(0,1)$ with the following property. If $(u,p)$ be a suitable solution to  the Navier-Stokes equations  in $Q_1$ such that 
$$A((0,0),1)+B((0,0),1)+  \int_{-1}^{0}\norm{p}_{L^{-\sigma,2}(B_1(0))} dt\leq \epsilon,$$
then $u$ is regular in $Q_{1/2}$. 
 \end{corollary}

Finally, we prove Theorem \ref{PhThm-sigma}.
\begin{proof}[Proof of Theorem \ref{PhThm-sigma}] By H\"older's inequality it follows that 
$$\int_{-1}^{0}\norm{p}_{L^{-\sigma,2}(B_1(0))} dt\leq D_\sigma((0,0),1)^{\frac{2-\sigma}{2}}.$$
 
Thus by Corollary \ref{L1-sigma}, Proposition \ref{ABcontrol0}, and a covering argument we obtain Theorem \ref{PhThm-sigma}.
\end{proof}

\end{document}